\documentclass[11pt,reqno,a4paper]{article}
\usepackage{amsmath}
\usepackage{amsbsy}
\usepackage{amsthm}
\usepackage{amssymb}
\usepackage{amscd}
\usepackage{mathabx}
\usepackage{accents}
\usepackage{float}
\usepackage{url}

\usepackage{mathabx}

\usepackage{hyperref}
\hypersetup{
pdftitle={},%
pdfauthor={},%
pdfsubject={},%
pdfkeywords={},%
colorlinks=true,%
linkcolor={black},%
linktoc={},
linktocpage={},%
pageanchor={},
citecolor={black},
}

\usepackage[english]{babel}
\usepackage[utf8]{inputenc}
\usepackage[margin=2.41cm]{geometry}
\usepackage{mathrsfs}
\usepackage[all]{xy}
\usepackage{hyperref}
\usepackage{url}

\usepackage{titlesec}
\titleformat{\section}{\large\bfseries\filcenter}{\thesection}{1em}{}
\titleformat{\subsection}{\bfseries}{\thesubsection}{1em}{}

\usepackage{etoolbox}

\makeatletter
\patchcmd{\ttlh@hang}{\parindent\z@}{\parindent\z@\leavevmode}{}{}
\patchcmd{\ttlh@hang}{\noindent}{}{}{}
\makeatother

\newtheorem{thm}{Theorem}[section]
\newtheorem{cor}[thm]{Corollary}
\newtheorem{lemma}[thm]{Lemma}

\newtheorem{prop}[thm]{Proposition}

\theoremstyle{remark}

\theoremstyle{definition}

\newtheorem{rmk}[thm]{Remark}
\newtheorem{defn}[thm]{Definition}

\newtheorem{notation}[thm]{Notation}

\numberwithin{equation}{section}
\allowdisplaybreaks[1]

\catcode`@=12

\def\beq{\begin{equation}}
\def\eeq{\end{equation}}

\def\beqn{\begin{equation*}}
\def\eeqn{\end{equation*}}

\def\ben{\begin{enumerate}}
\def\een{\end{enumerate}}

\renewcommand\thanks[1]{%
  \begingroup
  \renewcommand\thefootnote{}\footnote{#1}%
  \addtocounter{footnote}{-1}%
  \endgroup
}

\renewcommand{\epsilon}{{\varepsilon}}

\renewcommand{\epsilon}{{\varepsilon}}

\def\aa{{\mathfrak a}}
\def\0{{(0,0)}}
\def\C{{\mathbb C}}
\def\N{{\mathbb N}}
\def\R{{\mathbb R}}

\def\Z{{\mathbb Z}}
\def\cf{\emph{cf.}~}
\def\ie{\emph{i.e.}~}
\def\cA{{\mathcal A}}
\def\cE{{\mathcal E}}

\def\cG{{\mathcal G}}

\def\cV{{\mathcal V}}
\def\sI{{\mathscr I}}
\def\bA{{\mathbf A}}
\def\bH{{\mathbf H}}
\def\bP{{\mathbf P}}
\def\bR{{\mathbf R}}
\def\fA{{\mathfrak A}}
\def\fa{{\mathfrak a}}
\def\a{\alpha}
\def\Sp{{\rm Sp \,}}
\def\ol{\overline}
\def\Id{{\rm Id\,}}
\def\Sl{{\rm Sl\,}}
 
\renewcommand{\;}{\,;\,}
\begin{document}

\begin{flushright}

\baselineskip=4pt

\end{flushright}

\begin{center}
\vspace{5mm}

{\Large\bf  INVARIANT STATES ON  NONCOMMUTATIVE TORI}

\thanks{F. B. is supported by the DFG SBF/CRC 1085 { ``Higher Invariants. Interactions between Arithmetic Geometry and Global Analysis" }.
S. M. is supported by the research grant { ``Geometric boundary value problems for the Dirac operator''} and he was partially supported within the DFG research training group GRK 1692 { ``Curvature, Cycles, and Cohomology''}. 
N.P. thanks the ITP of the University of Leipzig for the kind hospitality during the preparation of this work and DAAD for supporting this stay with the program ``Research Stays for Academics 2017".}

\vspace{5mm}

{\bf by}

\vspace{5mm}

 { \bf Federico Bambozzi}\\[1mm]
\noindent  {\it Fakult\"at f\"ur Mathematik, }{\it Universit\"at Regensburg, } {\it D-93040 Regensburg, Germany}\\[1mm]
email: \ {\tt federico.bambozzi@mathematik.uni-regensburg.de}
\\[6mm]
{  \bf Simone Murro}\\[1mm]
\noindent  {\it Mathematisches Institute, }{\it Universit\"at Freiburg, } {\it D-79104 Freiburg, Germany}\\[1mm]
email: \ {\tt simone.murro@math.uni-freiburg.de}
\\[6mm]
{  \bf  Nicola Pinamonti}\\[1mm]
\noindent   {\it Dipartimento di Matematica, }{\it Universit\`a di Genova, }
{\it I-16146 Genova, Italy}\\
{\it
INFN - Sezione di Genova, I-16146 Genova, Italy}
\\[1mm]
email: \ {\tt  pinamont@dima.unige.it}
\\[8mm]
\today
\\[10mm]
\end{center}

\begin{abstract}
For any number $h$ such that $\hbar:=h/2\pi$ is irrational and  { any skew-symmetric, non-degenerate bilinear form $\sigma:\Z^{2g}\times \Z^{2g} \to \Z$, let be $\cA^h_{g,\sigma}$ be the twisted group $*$-algebra $\C[\Z^{2g}]$ and consider the ergodic group of $*$-automorphisms of $\cA^h_{g,\sigma}$ induced by the action of the symplectic group $\Sp(\Z^{2g},\sigma)$. } We show that the only $\Sp(\Z^{2g},\sigma)$-invariant state on $\cA^h_{g,\sigma}$ is the trace state $\tau$. 
\end{abstract}

\\[0.5mm]

\renewcommand{\thefootnote}{\arabic{footnote}}
\setcounter{footnote}{0}

\section*{Introduction}

Let $\fA$ be a unital $C^*$-algebra, $\cG$ a compact group and $\Phi$ a strongly continuous representation of $\cG$ as an ergodic group of $*$-automorphisms of $\fA$, i.e. 
$\Phi_\Theta (\fa) = \fa$ for all $\Theta \in \cG$ implies $\fa=\lambda \Id$, for some scalar $\lambda$, where $\Id$ is the identity of $\fA$. It was shown in \cite{Stormer} that if
$\cG$ is Abelian and $\fA$ a von Neumann algebra then the unique $\cG$-invariant state on $\fA$ is a trace state. For several years, it has been an open problem if the same result holds with weaker assumptions, see e.g. \cite{OPK}. 
An important step forward was made in \cite{HLS},  where it was shown that if $\cG$ is a compact ergodic group of automorphisms acting on a unital $C^*$-algebra $\fA$, the unique $\cG$-invariant state is a trace. \\
In most of the models inspired by mathematical physics, the ergodic group of $*$-automorphisms is neither compact nor Abelian (see e.g. \cite{Bellissard, DMS} where $\cG$ is considered to be only locally compact) and, therefore, it would be desirable to classify all the $\cG$-invariant states.  {Indeed, from a mathematical perspective, they provide a ‘noncommutative generalization’ of the invariant measures in ergodic theory. Moreover, the representations of the $C^*$-algebra are implemented by a unitary representation of $\cG$ acting on a Hilbert space. Instead, from a physical perspective, they represent equilibrium states in statistical mechanics ~\cite{Araki,DKS,HHW,KR}.}

 In the present paper we consider any dynamical system of the form  {  $(\cA_{g,\sigma}^h,$ $\Sp(\Z^{2g},\sigma),\Phi)$, where $g\in\N$, $h$ is any number such that $\hbar=\frac{h}{2\pi}$ is irrational, $\sigma$ is any skew-symmetric, non-degenerate bilinear form on $\Z^{2g}$,  $\cA_{g,\sigma}^h$ are twisted group $*$-algebras for $\Z^{2g}$, $\Sp(\Z^{2g},\sigma)$ is the subgroup of Aut$(\Z^{2g})$ which preserve  $\sigma$ while $\Phi$ are representations as ergodic groups of the $*$-automorphisms of $\cA_g^h$.} We prove that the trace state $\tau$ defined in~\eqref{tracialstate} is the unique $\Sp(2g,\Z)$-invariant state. 
The key idea is to construct, for any positive state $\omega$ different from a trace state, a convex linear combination of restrictions of $\omega$ to suitable finite dimensional subspaces of $\cA_{g,\sigma}^h$ that results not positive.  The choice of the subspaces of  {$\cA_{g,\sigma}^h$} is made so that for any two subspaces $\cV_1,\cV_2 \subset$   {$\cA_{g,\sigma}^h$} there exists  a $*$-automorphism $\Phi$ such that $\Phi(\cV_1)=\cV_2$. \\
Let us remark that for $g=1$ the twisted group $*$-algebra $\cA^h_1,\sigma$ can be completed to the universal $C^*$-algebra known as noncommutative torus. 
In this setting, it has been shown that, for almost all deformation parameters, the trace state is the unique invariant state with respect to any fixed hyperbolic element of SL$(2\Z) $, see e.g. ~\cite{IS1, IS2} or ~\cite[Section 11.5]{IS3}.
  In this paper, we prove that the trace state is the unique $ \Sp (2g, \Z)$-invariant state for all irrational deformation parameters. Our proof uses purely algebraic methods and thus it can be used to study invariant states on more sophisticated twisted group $*$-algebras with symplectic forms valued in abelian groups (see e.g.~\cite{spinoff}).

The paper is structured as follows. In the first section, we recall the definition of  { twisted group $*$-algebra and we study the orbits of the action of $\Sp(\Z^{2g},\sigma)$} on $\Z^{2g}$. Section 2 is the core of the paper, where the main theorem is stated and proved. This is achieved using reduction steps. In the first step, we show that, given any finite dimensional subspace of  {$\cA_{g,\sigma}^h$}, it is possible to associate to any state $\omega$ a positive Hermitian matrix $\bH$ (see Notation \ref{notation}). Then, a convex linear combination of restrictions of $\omega$ on different subspaces of  {$\cA_{g,\sigma}^h$} is constructed in order to prescribe the value of $\omega$ on a given  {$\Sp(\Z^{2g},\sigma)$} orbit (see Proposition \ref{bR}). As already explained, the construction of these subspaces of  {$\Sp(\Z^{2g},\sigma)$} cannot be generic, but it should preserve the  {$\Sp(\Z^{2g},\sigma)$}-invariance of the convex combination of the restrictions of $\omega$ previously considered. Finally, it will be shown that the convex linear combination discussed above 
 is positive if and only if $\omega$ is the trace state, proving our main result.
This result will be achieved by showing that the matrices $\bH$ obtained restricting the state $\omega$ to suitable subspaces of  {$\cA_{g,\sigma}^h$} can be approximated by simpler ones (see Proposition \ref{prop:ergodic}).\\

\textit{Acknowledgements.} 
We would like to thank Ulrich Bunke, Nicol\'o Drago, Francesco Fidaleo, Giuseppe De Nittis, Alexander Schenkel and Stefan Waldmann for helpful discussions related to the topic of this paper. We are also grateful to Claudio Dappiaggi, Emilia Mu\~noz {and to the referees} for their useful comments on the manuscript.

\section{ {Algebraic noncommutative tori}}
Let $h$ be a number such that $\hbar = \frac{h}{2 \pi}$ is irrational and consider  {a skew-symmetric, non-degenerate bilinear form $\sigma: \Z^{2g}\times \Z^{2g}\to \Z$. }
Let now $C(\Z^{2g},\mathbb{C})$ be the set of complex valued functions over $\Z^{2g}$. 
Consider $W_m$ the linear operator labelled by an element $m\in\Z^{2g}$ which acts on 
${ v}\in C(\Z^{2g},\mathbb{C})$ in the following way
\[
(W_m v)(n) :=  e^{\i h {\sigma(m,n)}} v(n + m).
\]
The complex vector space $\cV(\Z^{2g},\sigma)$ generated by the elements of $\{W_m,m\in \Z^{2g}\}$ can be endowed with an involution defined by
\begin{equation}\label{involution}
W_m^*= W_{-m} \,, \qquad m \in \Z^{2g} 
\end{equation}
and with a product which acts on the generators as  
\begin{equation}\label{WeylRel}
W_n W_m = e^{\i h \sigma(n,m)} W_{n + m}\,,  \qquad  n,m \in \Z^{2g} \,. 
\end{equation}
\begin{rmk}
 { Notice that being $\sigma$ bilinear, skew-symmetric and non-degenerate,  $\Omega:= \exp(\imath h \sigma)$ defines a group $2$-cocycle. Indeed, for any $0,m,n,g\in\cG$ we have $\Omega(0,m)=\Omega(m,0)=0$ and 
\begin{align*}
\Omega(m, n) \times \Omega(m + n, g) &= \Omega(m, n )\times \big(\Omega(m, g) \times\Omega(n, g) \big)= \\
&= \big( \Omega(m, n )\times \Omega(m, g) \big) \times\Omega(n, g) = \Omega(m, n + g)\times \Omega(n, g)\,.
\end{align*}
The converse is also true: as shown in \cite[Theorem 7.1]{Kleppner}, for any 2-cocycle $\Omega$ there exists a skew-symmetric bilinear form $\sigma:\Z^{2g}\times \Z^{2g}\to \frac{\R}{\Z}$ which is cohomologous to $\Omega$. Whereas in this paper we will only consider $\Z$-valued bilinear form, we refer to \cite{spinoff} for a more general dissertation.
}
\end{rmk}
\noindent We finally notice that any element $\aa\in \cV(\Z^{2g},\sigma)$ can be written as a finite linear combination$$ 
\aa=\sum_m \alpha_m W_{m} \,, \qquad\qquad  m\in\Z^{2g} \,.
$$
We are now ready to summarize this short discussion in the following definition.

\begin{defn} \label{defn:weyl_algebra}
 { We call \emph{algebraic noncommutative torus} $\cA^h_{g,\sigma}$ the $*$-algebra obtained equipping $\cV(\Z^{2g},\sigma)$ with the involution \eqref{involution} and with the product \eqref{WeylRel}.}   
\end{defn}

\begin{rmk}
 {When completed by a canonical $C^*$-norm, the twisted group algebras $\cA_{g,\sigma}^h$ are also called Weyl $C^*$-algebras or exponential Weyl algebras in the literature, see e.g. \cite{Moretti,Robinson1,Robinson2}. These algebras should not be confused with quotients of the universal enveloping algebras of the Heisenberg Lie algebra, obtained by identifying the central elements of the Heisenberg Lie algebra with multiples of the identity element, which are also called Weyl algebras, see e.g. \cite{Dixmier1,Dixmier2}.}
\end{rmk}

\noindent Let us underline that 
the action of  {$\Sp(\Z^{2 g},\sigma)$} can be extended by linearity to an algebra automorphism of $\cA^h_{g,\sigma}$ as 
\begin{equation}\label{eq:SPaction} 
(\Phi_\Theta W)_m = W_{\Theta m}, \qquad \text{ for } m\in \Z^{2g } \text{ and }\Theta \in \Sp(\Z^{2g},\sigma). 
\end{equation}

\begin{lemma}\label{reduction}
 {Let $\sigma$ be a skew-symmetric, non-degenerate form on $\Z^{2g}$ and  let  $m\in\Z^{2g}$. Then there exists
$\delta_1,\dots,\delta_n \in \Z$ such that
$$ (\Z^{2g},\sigma)\simeq (\Z^2,\delta_1\sigma_2) \oplus \dots  \oplus (\Z^2,\delta_n\sigma_2) \,,$$
being $\sigma_2$  the canonical symplectic form on $\Z^2$ given by
$$\sigma_2=\begin{pmatrix}
0 & -1 \\ 1 & 0
\end{pmatrix}\,.$$}
\end{lemma}
\begin{proof}
 {This lemma follows easily from  \cite[Theorem IV.1]{Morris}.}
\end{proof}

 {On account on Lemma~\ref{reduction}, we can  focus on the algebraic noncommutative torus $\cA^h:=\cA^h_{g=2,\sigma_2}$ without loss of generality.}

\begin{notation}
	We remark that even if the isomorphism class of the algebra $\cA^h_{g,\sigma}$ depends on the choice of the number $h$, all our results are independent of this choice, provided $\hbar$ is kept irrational. Therefore, we omit to refer to it in the notation. 
\end{notation}


In the next proposition we establish a one-to-one correspondence between the orbits of the symplectic group $\Sp(2,\Z)$ and the set of elements
$$  
\cE :=\{ (0,j)\, |\, j \in \N \}\,.
$$

\begin{prop} \label{prop:orbit}
	Let $n_1,n_2\in\Z^{2}$ be of the form $n_i= (0,m_i)$, with $m_i\in \Z$. If $m_i \geq 0$ and $m_1 \neq m_2$, then $n_1$ and $n_2$ are elements of different orbits. 
	Furthermore, for any element  $n \in \Z^{2}$, there exists a $\Theta\in \Sp(2,\Z)$ such that 
	$\Theta n=(0,m)$ for some $m\in \mathbb{N}$. 
\end{prop}

\begin{proof}
	Since $m_1\neq m_2$, we can set $m_1\neq 0$. Assume that there exists $\Theta\in \Sp(2,\Z)$ such that $\Theta n_1=n_2$. Since a generic $\Theta$ takes the form 
	\[
	\Theta = \begin{pmatrix}  a & b \\ c & d  \end{pmatrix}, \qquad a,b,c,d \in \mathbb{Z}, 
	\]
	$\Theta n_1=n_2$ implies $b=0$. Furthermore, since $\Sp(2,\Z)=\Sl(2,\Z)$, then $\det\Theta=1$. As a consequence,  $a=d$ and $a=\pm 1$. Hence, since both $m_1$ and $m_2$ are positive, $\Theta n_1 = n_2$  implies that $\Theta$ is the identity. This contradicts the fact that $m_1\neq m_2$.
	
	Consider now the action of $\Theta \in \Sp(2,\mathbb{Z})$ on a generic element $n=(n_1,n_2)\in\Z^2$, namely  $\Theta v = (an_1+bn_2,cn_1+dn_2)$. We are going to demonstrate that it is possible to select $a,b,c,d$ such that $an_1+bn_2=0$ and $ad-bc=1$. In the case in which $n_1$ or $n_2$ is $0$, it is enough to set $b=0,$ $a=d=\pm 1$ or $a=0,$ $b=-c=\pm1$ respectively. If $n_1=n_2\neq 0$, we just choose $a=1,b=-1,d=1,c=0$. It remains to deal with the case in which $n_1\neq n_2$ and both are non-zero. To this end, we denote $g=\textnormal{g.c.d.}(n_1,n_2)$, the greatest common divisor of $n_1$ and $n_2$. Thus $n_i=e_i g$ for a couple of integers $e_i\in\mathbb{Z}$ which are thus coprime.
	We can choose,   
	$a=e_2$, $b=-e_1$, in order for the first equation $an_1+bn_2=0$ to be satisfied. It remains to show that there exists a choice of $c,d\in \mathbb{Z}$ such that
	\[
	ac-bd = 1  \Longrightarrow   e_2 c + d e_1  = 1.
 	\]
	Suppose that $e_2 < e_1$ and consider $\Z_{e_1}:=\mathbb{Z}\slash (e_1 \mathbb{Z})$. Since $e_1$ and $e_2$ are coprime, then
	$e_2 \in (\Z_{e_1})^\times$ which is the multiplicative group (modulo $e_1$) formed by the subset of elements of $\mathbb{Z}\slash (e_1 \mathbb{Z})$ coprime to $e_2$. 
	Hence, $e_2$ has an inverse in $\mathbb{Z}_{e_1}$ and, therefore, it is possible to find $c$ and $d$ such that 	$e_2 c + d e_1  = 1$.
	Since the case $e_1<e_2$ can be obtained interchanging the role of $n_1$ and $n_2$, this concludes our proof.
\end{proof}

\begin{cor} \label{cor:orbit}
	Every $\Sp(2,\Z)$-orbit of $\Z^2$ contains an element of the form $(j, j)$ with $j \in \N$.
\end{cor}

\begin{proof}
	Consider the element $(0, j)$ that by Proposition \ref{prop:orbit} can be found in any $\Sp(2,\Z)$-orbit of $\Z^2$. Then the element
	\[ \begin{pmatrix}  1 & 1 \\ 0 & 1  \end{pmatrix} \cdot \begin{pmatrix}  0 \\ j  \end{pmatrix} = \begin{pmatrix}  j \\ j  \end{pmatrix} \]
	belongs to the same orbit.
\end{proof}

We conclude this section with the following remark. 
\begin{rmk}
Since the set of points which are fixed under the action of $\Sp(2,\Z)$ on $\Z^2$ is $\{(0,0)\in \Z^2\}$, the action of $\Phi_\Theta$, with $\Theta\in \Sp(2,\Z)$, is \textit{ergodic} on $\cA$, i.e. the vectors $\lambda W_{(0,0)}$, with $\lambda \in \C$, are the only invariant elements in $\cA$. Notice that $W_\0$ is indeed the identity of the  {algebraic noncommutative torus}.
\end{rmk}

\section{Invariant states}\label{InvariantState}

Let now $\omega$ be a state, namely a linear, continuous functional  from $\cA $ into $\C$ that is positive (i.e. $\omega(\aa^*\aa)\geq 0$ for any $\aa\in\cA $) and normalized (i.e. $\omega(W_{(0,0)})=1 \,).$ 

\begin{defn} \label{defn:invariant_state}
We call a state $\omega$ on $\cA $ \emph{$\Sp(2,\Z)$-invariant} if for any $*$-automorphism $\Phi_\Theta$, with $\Theta\in \Sp(2,\Z)$, it holds
	\[  \omega \circ \Phi_\Theta =  \omega \,.   \] 
\end{defn}
In order to construct a state $\omega$ on $\cA $ it is enough to prescribe its values on the generators $W_m$, $m\in \Z^2$,
\begin{equation}\label{eq:state on W}
\omega(W_m)=\begin{cases} 1 & \text{ if } m=\0 \\ p^{(m)} \in \C & \text{ else} \end{cases}
\end{equation}
for a sequence of values $p^{(m)}$
and then extend it by linearity to any element $\aa\in\cA $. 
We shall see below in Remark \ref{rmk:p<=1} that the positivity of $\omega$ implies    
\[ 
\sup_{m \in \Z^2} |p^{(m)}| \leq 1. 
\]
The theorem below, which is the main result of this paper, shows that the only $\Sp(2,\Z)$-invariant state is the trace state.

\begin{thm} \label{thm:main}
Let $\cA $ be an  {algebraic noncommutative torus}. Then the only $\Sp(2,\Z)$-invariant state is the state defined for every $m\in\Z^2$ as
	\begin{equation}\label{tracialstate}
\tau(W_m)=\begin{cases} 1 & \text{ if } m=(0,0) \\ 0 & \text{ else .} \end{cases}
\end{equation}
\end{thm}

The rest of this section is devoted to prove Theorem \ref{thm:main}. Given a state $\omega$ written as \eqref{eq:state on W}, our first observation is the following.

\begin{prop}\label{prop:real}
Let $\cA $ be an  {algebraic noncommutative torus} and consider a $\Sp(2,\Z)$-invariant state $\omega$. Then, for any $m\in \Z^2$, it holds
$$\omega(W_m) \in \R \,.$$
\end{prop}
\begin{proof}
Since $\omega$ is a linear positive functional, then, for every $m\in\Z^2$, it holds
\begin{align*}
\omega\left( (W_m + W_{(0,0)})^*(W_m + W_{(0,0)})\right) =   2  +\omega\left(W_{m}^*\right) + \omega\left(W_m\right)  \in [0,\infty)  \,.
\end{align*}
This implies in particular that $\ol{\omega(W_m)} = \omega(W_m^*)$.\\ Now,
let $\Id$ be the $2$ by $2$ identity matrix and notice that $-\Id \in \Sp(2,\Z)$. It follows that for every $m\in \Z^2$ we have
	\[
	\ol{\omega(W_m)} = \omega(W_m^*) =\omega(W_{-m})=\omega(W_{-\Id m}) = \omega(W_m)   
	\]
	where in the fourth equality we used the invariance of the state under the action of the symplectic group.
	\end{proof}

Now let $m,n \in \N$ be such that $m$ is divisible by $n \ge 1$, and consider the subset $\mathcal{G}_{m,n}$ of $\Sp(2,\Z)$  containing the elements of the form
\begin{equation}\label{Theta form}
\Theta_j  := \begin{pmatrix}
1& \frac{m}{n} j \\ 0 & 1 
\end{pmatrix}\cdot\begin{pmatrix}
1 & 0\\ n-1 & 1 
\end{pmatrix} = \begin{pmatrix}
1 + \frac{m}{n}(n-1)j & \frac{m}{n} j \\ n-1 & 1 
\end{pmatrix}, \qquad  j\in \mathbb{Z}\,.
\end{equation}	
 
Let $\xi:=(\xi_1,\xi_2)$ be an element of $\Z^2$ with $\xi_1 = \xi_2 > 0$ (thanks to Corollary \ref{cor:orbit} this choice is not restrictive) and
consider the set 
$$\mathcal{O}_{\xi;\, m, n} =\{ z\in\Z^2 \,|\, z=\Theta_j \xi \,, \text{ with } \Theta_j \in \mathcal{G}_{m,n} \}\,.$$
Notice that for all $n, m$ as above the elements of $\mathcal{O}_{\xi;\, m, n} $ belongs to the same $\Sp(2, \Z)$-orbit of $\xi$ and take the form
$$ \Theta_j \xi := \begin{pmatrix}
1 + \frac{m}{n}(n-1)j & \frac{m}{n} j \\ n-1 & 1 
\end{pmatrix} \begin{pmatrix}
\xi_2 \\ \xi_2 
\end{pmatrix} = \begin{pmatrix}
 m j \xi_2 + \xi_2 \\  n \xi_2 
\end{pmatrix}. $$
Let $\mathcal{V}_{\xi;\, m, n} \subset \cA $ be the vector space formed by the linear combinations of the identity in $\cA $ and of the Weyl generators indexed by the elements of $\mathcal{O}_{\xi;\, m,n}$. In other words, a generic element of $\mathcal{V}_{\xi;\, m, n}$ can be written as a finite sum of the form
\begin{equation*} 
\aa = \alpha_0 W_\0+ \sum_{j\geq 1} \alpha_j W_{\Theta_j\xi} 
\end{equation*}
where $\alpha_i\in\mathbb{C}$, $\Theta_j \in \cG_{m,n}$. Since we are interested in $\Sp(2,\Z)$-invariant states $\omega$, we have
$$\omega(\aa) = \alpha_0+\sum_{j\geq 1} \alpha_j\, p \, ,$$
where we denoted $p = \omega(W_{\Theta_j \xi})$, which does not depend on $j, m, n$ in view of the state invariance.
Notice that, for every finite dimensional subspace of $\cV_{\xi;\, m, n}$, the map $\aa \mapsto \omega(\aa^*\aa)$ is a quadratic form; therefore, it can be written as
\begin{equation}\label{eq:Homega}
\omega(\aa^*\aa) = \overline{\a}^t \, \bH \, \a
\end{equation}
for an Hermitian matrix and $\a$ a vector with components $\a_j\in\C$. On the $(d+1)$-dimensional subspace spanned by the elements $\{ W_{(0, 0)}, W_{\Theta_{j} \xi} \}_{1 \le j \le d}$ the entries of the matrix $\bH$ can be described as 
\begin{align} \label{eq:H-entries1}
(\bH)_{0,0}&=(\bH)_{j,j} = 1 , \qquad &  d\geq j \geq 1\\
\label{eq:H-entries2}
(\bH)_{0,j} &= (\bH)_{j,0} = p , \qquad &  d \ge j\geq 1
\\
(\bH)_{j,i} &=  q_{{(i-j)m}} e^{\i (i-j) \phi_{m,n}},   \qquad & d \ge i > j  \geq 1  
\label{eq:H-entries3}
\end{align}
where \eqref{eq:H-entries1} holds because of the state normalization condition,
\begin{equation}\label{eq:H-component}
 q_{(i-j)m}  := \omega(W_{ \Theta_{i} \xi- \Theta_{j} \xi}),
\qquad
h \sigma(\Theta_{i} \xi , \Theta_{j} \xi)= (i-j)\phi_{m,n} \,
\end{equation}
and
\begin{equation}\label{eq:phase}
\phi_{m,n} := h m n \xi_2^2. 
\end{equation}
Notice that $p,$ $q_{(i-j)m}$ and $\phi_{(i-j)m,n}$ are real numbers; in particular $q_{(i-j)m}$ is a real number thanks to Proposition~\ref{prop:real}.

\begin{notation}\label{notation}
We remark that, on account of equations \eqref{eq:H-entries1}, \eqref{eq:H-entries2} and \eqref{eq:H-entries3}, the Hermitian matrix 
\[
\bH= \begin{pmatrix}
	1 & p & p & p & p & p  & \ldots \\
 	p & 1 & q_m e^{\i \phi_{m,n}} &   q_{2m} e^{2\i \phi_{m,n}}  &  q_{3m} e^{3\i \phi_{m,n}} &  q_{4m} e^{4\i \phi_{m,n}}    & \ldots  \\
	p & q_{m} e^{-\i \phi_{m,n}}  & 1& q_{m} e^{\i \phi_{m,n}}  &    q_{2m} e^{2\i \phi_{m,n}}  &  q_{3m} e^{3\i \phi_{m,n}}   & \ldots \\
	p & q_{2m} e^{-2\i \phi_{m,n}}  & q_{m} e^{-\i \phi_{m,n}}  & 1& q_{m} e^{\i \phi_{m,n}}  &    q_{2m} e^{2\i \phi_{m,n}}   & \ldots \\
	p & q_{3m} e^{-3\i \phi_{m,n}}  & q_{2m} e^{-2\i \phi_{m,n}} & q_{m} e^{-\i \phi_{m,n}}  & 1& q_{m} e^{\i \phi_{mn}}   & \ldots \\
		p & q_{4m} e^{-4\i \phi_{m,n}}  &  q_{3m} e^{-3\i \phi_{m,n}}  &  q_{2m} e^{-2\i \phi_{m,n}}  & q_{m} e^{-\i \phi_{m,n}} & 1 & \ldots \\
	\vdots & \vdots &\vdots&\vdots&\vdots&\vdots& \ddots \\
	\end{pmatrix}
\]
is completely determined once the second row is known. Therefore, in order to keep simple our notation, we denote $\bH$  simply as
$$\bH=[ \, p \; 1 \; q_m e^{\i \phi_{m,n}} \; q_{2m} e^{\i 2 \phi_{m,n}} \; q_{3m} e^{\i 3 \phi_{m,n}} \;\dots]\,.$$
\end{notation}

\begin{rmk} \label{rmk:consistency}
Notice that the notation of the matrices $\bH$ is consistent with respect to any choice of $m, n$ as done previously. Indeed, in view of its definition \eqref{eq:H-component},
$q_{(i-j)m}$
is a function of
\[
\Theta_{i} \xi- \Theta_{j} \xi =
\begin{pmatrix}
(1 + i m)\xi_2 \\ n \xi_2
\end{pmatrix}
- 
\begin{pmatrix}
(1 + j m)\xi_2 \\ n \xi_2
\end{pmatrix}
=
\begin{pmatrix}
(i - j)m \xi_2 \\ 0
\end{pmatrix}, 
\]
which does not depend on $n$ but depends only on $(i - j)m$, which is the subscript of $q$. We will constantly exploit this fact in order to compare the entries of matrices obtained by restricting $\omega$ to different finite dimensional subspaces of $\cV_{\xi; \, m,n}$. Hence, once the notation defined so far is used, in order to check that two entries in two different matrices $q_r$ and $q_t$ agree, it will suffice to check that $r = t$.
\end{rmk}

Building on Remark \ref{rmk:consistency}, we notice that, in the entries of the matrices defined so far, the $q_{(i-j)m}$ do not depend on $n$, whereas the arguments $\phi_{m,n}$ do. This fact will play a crucial role in the analysis that will follow.  

\begin{rmk}\label{rmk:p<=1}
We notice immediately that the positivity of the state $\omega$ gives a bound on $p$ in  $\bH$. Actually, since the determinant of the upper left $2\times 2$ sub matrix of $\bH$ is positive, we have that 
\[
1-p^2\leq0,
\qquad
\Longrightarrow
\qquad
|p|\leq 1.
\]   
\end{rmk}

We begin by discussing the form of the families of matrices of restrictions of $\omega$ that we seeking. 
Let $d$ be a fixed natural number, and consider the $(d+1)\times (d+1)$ matrix 
$\bH''_n$
which is obtained restricting $\bH$ on a $d+1$ dimensional subspace of $\cV_{\xi;\,m,n}$. The next proposition shows that it is always possible to choose $m$ in such a way that $\bH_n''$ can be well approximated by the $(d+1)\times (d+1)$ matrix of the form
\begin{equation}\label{bH}
\bH'_{n} := 
[\,p \; 1 \; q_{m} e^{	 \i\frac{2 \pi}{d}n }\; q_{2m} e^{2\i  \frac{ 2\pi}{ d}n }\; q_{3m} e^{ 3\i \frac{2\pi}{d}n }\; \ldots 
\; q_{(d-1)m} e^{(d-1)\i  \frac{ 2\pi}{ d }n} ].
\end{equation} 
At this point, it is interesting to notice that in the context of noncommutative geometry the analysis of  approximations of algebras or states can shed light on some interesting hidden structures, see e.g. \cite{BKR,LLS}.

\begin{prop}\label{prop:ergodic}
Let $\xi\in\mathbb{Z}^2$; then
for any $d\in \mathbb{N}$ which is non zero, for every $l\in[1,d]$ and for every $\epsilon>0$, there exists an $\mathtt{N}\in\mathbb{N}$ which is a multiple of $d!$ and a $(d +1)$-dimensional subspace $\cV_{d;l} \subset \cV_{\xi;\, \mathtt{N}, l}$, such that the restriction of $\bH$ to $\cV_{d; l}$ is given by a $(d+1) \times (d+1)$-matrix of the form
\begin{equation}\label{bH_1}
\bH''_{l} = \bH'_{l} + \mathbf{1}_\epsilon + \mathbf{i}_\epsilon \,.  
\end{equation} 
Here $\bH'_{l}$ is given in \eqref{bH} with $m=\mathtt{N}$ and $n=l$
while  $\mathbf{1}_\epsilon$ and $\mathbf{i}_\epsilon$ are the Hermitian matrices with components
	 \begin{align*}  
&(\mathbf{1}_\epsilon)_{j,0} = (\mathbf{1}_\epsilon)_{j,j} = (\mathbf{1}_\epsilon)_{0,j}  = 0  &\text{ and }  \qquad &(\mathbf{i}_\epsilon)_{j,0} = (\mathbf{i}_\epsilon)_{j,j}= (\mathbf{i}_\epsilon)_{0,j} = 0  & \text{ for }   d\geq j\geq 0\\
&(\mathbf{1}_\epsilon)_{j,i} = \epsilon_{j,i}  & \text{ and }  \qquad & (\mathbf{i}_\epsilon)_{j,i} = \i\,\epsilon_{j,i}'  & \text{ otherwise }
\end{align*} 
for some $\epsilon_{j,i}, \epsilon_{j,i}' \in \R$ with $|\epsilon_{j,i}|, |\epsilon_{j,i}'| < \epsilon$.
\end{prop}
\begin{proof}
Since $\hbar:= \frac{h}{2\pi}$ is irrational, it follows that there exists an $\mathtt{N} \in \N$ such that
\begin{equation}\label{N grande}
 \left|(h\, \mathtt{N}\, \xi_2^2 ) \mod(2 \pi) - \frac{2\pi}{d}\right| < \frac{\epsilon}{4 d^2} \,. 
\end{equation}
By the ergodicity of the multiplication of $S^1$, we can find such an $\mathtt{N}$ satisfying $\mathtt{N}\pm d! \equiv \mathtt{N}$, because $\mathtt{N}$ can be chosen to be arbitrarily large. Therefore, we can assume $\mathtt{N}$ to be divisible by $d!$.\\
Now consider the subset $\cG_{\mathtt{N},l}$ of the $\Sp(2,\Z)$ containing the elements of the form given in \eqref{Theta form}, which we recall here
$$ 
\Theta_j  = \begin{pmatrix}
 1+ \frac{m (n-1)}{n} j & \frac{m}{n} j \\ n-1 & 1 
\end{pmatrix} \in \cG_{m,n}, \qquad \text{ with } \quad m:= \mathtt{N} \,,  \quad n:= l \,,\quad j\in \mathbb{Z}\,.$$
Consider now the set of the generators 
\[ S= \{ W_{(0,0)} \} \cup \left\lbrace  W_{ \left( \mathtt{N} j  \xi_2 + \xi_2, l \xi_2 \right)} \right\rbrace_{1 \le j \le d} \]
and denote their linear span as 
$$
\cV_{d; l} = \left \{\alpha_0 W_\0 + \sum_{j=1}^{d} \alpha_j W_{\Theta_{j }\xi} \,|\, \alpha_j \in \C\,, W_{\Theta_{j } \xi } \in S \right \} \,.
$$ 
Then the restriction of $\omega$ to $\cV_{d; l}$ is described as in \eqref{eq:Homega} by a Hermitian matrix of the form
\begin{align*} \label{eq:H-entries}
(\bH'')_{0,0} &=(\bH'')_{j,j}= 1 , \qquad & d\geq j\geq 1
\notag
\\
(\bH'')_{0,j} &= p , \qquad & d\geq j\geq 1
\notag
\\
(\bH'')_{j,i} &=  q_{(i-j)m} e^{\i (i-j) \phi_{m,n}},   \qquad &d\geq i > j  \geq 1,
\end{align*}
where, using \eqref{eq:H-component} and \eqref{eq:phase}, we get
\begin{gather*}
 q_{(i-j)m} = \omega(W_{ \Theta_{i } \xi- \Theta_{j} \xi}) = \omega(W_{((i - j)\mathtt{N} \xi_2 , 0 )})
\\
(i-j) \phi_{m,n} = h \sigma(\Theta_{j} \xi , \Theta_{i} \xi) = (i - j) l h \mathtt{N}\xi_2^2  = (i - j) l \frac{2\pi}{d}   + (i - j) l\frac{\epsilon}{4d^2} 
\,.
\end{gather*}
Notice that, being $q_m\leq 1$ by Remark \ref{rmk:p<=1}, we have that for $d\geq i > j  \geq 1$
\[
\left|(\bH'')_{j,i} -(\bH')_{j,i} \right|=\left|
q_{(i-j)m} \left( e^{\i (i-j) \phi_{m,n}} - e^{	 \i(i - j)\frac{2 \pi}{d}l }\right)
\right|\leq 
\left|
 e^{\i (i-j) l\epsilon }-1\right| < 2 |i-j| l\frac{\epsilon}{4d^2} < \epsilon.
\]
It is now straightforward to check that the matrix so defined satisfies the properties claimed in the proposition.
\end{proof}

Since $\epsilon$ can be chosen arbitrarily small, its contribution will not play any role in the following computations, and, therefore, it is neglected. Hence, it will be important to discuss in detail the properties of the matrices $\bH'_{l}$ only. 
Therefore, all the equalities that will be discussed from here on hold also substituting $\bH'_{l}$ with  $\bH''_l$ up to $\epsilon$ for any $\epsilon > 0$, and then their limit for $\epsilon \to 0$ gives the desired result. 
We shall see in the proof of Theorem \ref{thm:main} that 
this is in fact the case and, in particular, that the error can be easily estimated in terms of $\epsilon$.
The limit $\epsilon\to0$ poses then no problem because only finitely many equations are considered at each time. 

\begin{lemma}\label{lem:convex_cone}
The set of positive Hermitian matrices form a convex cone, \ie convex combinations of positive Hermitian matrices are positive. 
\end{lemma}
\begin{proof}
Let $\mathscr{I}$ be a finite set and consider a convex combination $\bA$ of positive definite $d\times d$-matrices $\{A_i\}_{i \in \sI}$, namely 
$$\bA=\sum_{i\in \mathscr{I}} \lambda_i A_i \qquad \text{ with } \quad \sum_{i\in \sI} \lambda_i = 1 \text{ and } \lambda_i >0 \,\, \forall i \in \mathscr{I} \,.$$
Then, for any vector $v\in \C^d$, we have
$$ v^\dagger \bA v =\sum_{i\in \mathscr{I}} \lambda_i \,v^\dagger   A_i v > 0 \,. \qedhere$$
\end{proof}

On account of Lemma \ref{lem:convex_cone}, if we can find a convex combination of $(d+1)\times(d+1)$-matrices given by restrictions of $\omega$ to $\mathcal{V}_{d;\, l}$ which is non-positive, then we can deduce that at least one of these matrices is not positive (proving that $\omega$ is not positive neither). The next series of lemmas is necessary to produce such a convex combination and to prove some of its properties. Let us call to mind that we keep the notation introduced in Proposition \ref{prop:ergodic}.

\begin{prop}\label{bR}
Let $d \in\N$ be such that $d > 0$ and consider the matrices $\bH'_{l}$ defined in \eqref{bH} and obtained in Proposition \ref{prop:ergodic} with $l\leq d$. Moreover, let us consider
\[ 
\bR_{d} =  \sum_{l = 1}^{d} \frac{1}{d} \bH'_{l} \,. 
\]
Then the components $R_{j,d}$ of the matrices 
$$ 
\bR_{d} = [\,p\;1\;R_{1, d}\;R_{2, d}\;R_{3, d}\;R_{4, d}\;\dots\; R_{d-1,d}\,] 
$$
are such that $R_{j, d} = 0$. 
\end{prop}
\begin{proof}
Notice that, on account of the form of $\bH_l'$ given in $\eqref{bH}$, it holds
\[
R_{j,d} = q_{j \mathtt{N}}   \sum_{l=1}^{d}  e^{ \i \frac{2\pi}{d} l j}
=q_{j \mathtt{N}}\sum_{l=1}^d r^l = q_{j \mathtt{N}} \frac{r}{1-r}(1-r^{d})\,,
\]
where in the third equality we computed the sum of $d$ elements of a geometric series of ratio $r=e^{\i\frac{2\pi}{d} j}$. Furthermore, the last equality holds because  $1\leq j<d$ and thus $r\neq1$. Finally, we notice that $r^{d}=1$ and accordingly $R_{j,d}$ vanishes.   
\end{proof}
Before giving the proof of Theorem \ref{thm:main}, we need another lemma involving the $(d+1)\times (d+1)$ matrix $\bP_d$ defined as 
\begin{equation}\label{eq:Pd}
\bP_d := [p;1;0;0;\dots ;0],
\end{equation}
which is nothing but $\bR_{d}$ given in Proposition \ref{bR}. 

\begin{lemma}\label{lem:Pn_negative}
Let $d\in\N$ and consider the $(d+1)\times (d+1)$ matrix $\bP_d$ defined in \eqref{eq:Pd}. Then, it holds 
\[
\det(\bP_d)= 1 - d p^2 \,. \qedhere
\]
\end{lemma}
\begin{proof}
We prove it by induction. By direct computations, we have that $\bP_1 = 1 - p^2$. Suppose that $\det{\bP_d}= 1 - d p^2$, we can compute 
$\det{\bP_{d+1}}$ by the Laplace formula expanding with respect to the last column. The resulting formula has two terms, the second of which can be computed directly and the first one is the determinant of ${\bP_d}$, hence
\[
\det{\bP_{d+1}}= \det{\bP_d} + (-1)^{d+1}  p^2 (-1)^{d} = 1 - {(d+1)} p^2 \,. \qedhere
\]
\end{proof}
We have now all the ingredients to prove the main Theorem.

\begin{proof}[Proof of Theorem \ref{thm:main}]
One notices immediately that the trace state $\tau$ is positive, normalized and $\Sp(2,\Z)$-invariant. Now we assume that $\omega$ is a positive $\Sp(2,\Z)$-invariant state which is not the trace state. As explained so far, $\omega$ can be uniquely determined by
$$
\omega(W_m)=\begin{cases} 1 & \text{ if } m=\0 \\ p^{(m)} & \text{ else ,} \end{cases}
$$
where $W_m$ is a generator of $\cA$ and hence at least one $p^{(m)}$ is a non-zero real number (\cf Proposition \ref{prop:real}). Consider one such $m \in \Z^2$ for which $p^{(m)} \ne 0$. By Corollary \ref{cor:orbit} we know that $\omega(W_m) = \omega(W_\xi)$ for a $\xi = (j,j)$ given by a unique $j \in \N$. To such a $\xi$ we can always associate the subspaces $\cV_{\xi; \, m,n}$ of $\cA$ and the matrices $\bH_{l}', \bH_{l}''$ obtained by applying Proposition \ref{prop:ergodic}, for which now $p = p^{(m)}$. Then, applying Proposition \ref{bR}, we obtain a convex combination $\bP_d$ of the matrices $\bH_{l}'$ such that 
\[ 
\bP_d = [p;1;0;0; \ldots; 0] 
\]
and, because of Lemma \ref{lem:Pn_negative}, the determinant of $\bP_d$ is $1 - d (p^{(m)})^2$. Since $d$ can be chosen to be an arbitrarily big positive integer for any fixed $p^{(m)} > 0$, we can find a $d$ such that $\det(\bP_d)$ is non-positive, and hence $\bP_d$ is not a positive matrix. Applying Lemma \ref{lem:convex_cone}, we deduce that, for all matrices $\bH_{l}'$ to be positive, it is necessary that $p^{(m)} = 0$, contradicting our hypothesis that $p^{(m)} \ne 0$. In order to deduce the same result for the matrices $\bH_{l}''$, it is enough to apply Proposition \ref{prop:ergodic} 
and notice that, if $\epsilon<1$, we have that  
\[
|\det(\bP_d') | <  |\det(\bP_d)| + \epsilon R(d), 
\]
where 
\[ \bP_d' =  \sum_{l = 1}^{d} \frac{1}{d} \bH''_{l} = \sum_{l = 1}^{d} \frac{1}{d} (\bH'_{l} + \mathbf{1}_\epsilon + \mathbf{i}_\epsilon ), 
\]
is the analog of $\bP_d$ where $\bH_{l}''$ is taken as input in place of $\bH_{l}'$ and $R(d)$ is a positive function of $d$ which bounds the reminder.
A non optimal estimate for $R(d)$ is 
\[
R(d)\leq 2 d(d-1) d!,
\] 
and it follows from the observation that in both 
$\mathbf{1}_\epsilon$ and $\mathbf{i}_\epsilon$ there are at most $d(d-1)$ non vanishing entries. Furthermore, in the Laplace expansion of the determinant of $\bP_d'$, each of these non vanishing entries is multiplied by the determinant of a square submatrix of $\bP_d'$. These submatrices are of at most $d\times d$ dimension. Moreover, in view of Remark \ref{rmk:p<=1}, the absolute value of the entries of $\bP'_d$ are bounded by $1$, and thus the determinant of each submatrix is bounded by $d!$.
Hence, since in Proposition \ref{prop:ergodic} the value of $\epsilon$ can be chosen freely, 
the difference between $|\det(\bP_d') |$ and  $|\det(\bP_d)|$ can be set to be arbitrarily small.
This proves that $p^{(m)} = 0$ is a necessary condition also for $\omega$ to be positive, as a consequence of Lemma \ref{lem:convex_cone}.
Since this is true for all non-zero $m \in \Z^2$, we have proved that $\omega$ is the trace state. 
\end{proof}

\end{document}